\documentclass{amsart}%
\usepackage{amsmath}
\usepackage{amsfonts}
\usepackage{amssymb}
\usepackage{graphicx}%

\theoremstyle{plain}
\newtheorem{theorem}{Theorem}
\newtheorem*{theorem*}{Theorem}
\newtheorem{proposition}{Proposition}
\newtheorem{lemma}{Lemma}

\newtheorem{definition}{Definition}

\theoremstyle{remark}

\newtheorem{problem}{Problem}

\numberwithin{equation}{section}
\allowdisplaybreaks

\begin{document}
\title[H.~Bohr's theorem]{H.~Bohr's theorem for bounded symmetric domains}
\author{Guy Roos}
\address{Nevski pr. 113/4-53, 191024 St Petersburg, Russian Federation}
\email{guy.roos@normalesup.org}
\date{April 9, 2009}

\begin{abstract}
A theorem of H.~Bohr (1914) states that if $f(z)=\sum_{k=0}^{\infty}a_{k}%
z^{k}$ is a holomorphic map from the unit disc $D\subset\mathbb{C}$ into
itself, then $\sum_{k=0}^{\infty}\left\vert a_{k}z^{k}\right\vert \leq1$ for
$\left\vert z\right\vert \leq\frac{1}{3}$; the value $\frac{1}{3}$ is optimal.
This result has been extended by Liu Taishun and Wang Jianfei (2007) to the
bounded symmetric domains of the four classical series, and to polydiscs. The
result of Liu and Wang may be generalized to all bounded symmetric domains,
with a proof which does not depend on classification.

\end{abstract}
\maketitle

\section*{Introduction}%

The following theorem was proved by Harald Bohr \cite{Bohr1914} for
$\left\vert z\right\vert <\frac{1}{6}$, then soon improved to $\left\vert
z\right\vert <\frac{1}{3}$ from remarks by M.~Riesz, I.~Schur,
F.~Wiener\footnote{See \cite{BoasKhavinson2000} for biographical elements
about Friedrich Wiener, not to be confused with Norbert Wiener.}:

\begin{theorem}
\label{HBohrTheorem}Let $f:\Delta\rightarrow\Delta$ be a holomorphic function
from the unit disc $\Delta\subset\mathbb{C}$ into itself. If $f(z)=\sum
_{k=0}^{\infty}a_{k}z^{k}$, then
\begin{equation}
\sum_{k=0}^{\infty}\left\vert a_{k}z^{k}\right\vert <1 \label{eq0}%
\end{equation}
for $\left\vert z\right\vert <\frac{1}{3}$. The value $\frac{1}{3}$ is optimal.
\end{theorem}

This result has been extended by Liu Taishun and Wang Jianfei
\cite{LiuWang2007} to the bounded symmetric domains of the four classical
series, and to polydiscs, using a case-by-case analysis, as follows:

\begin{theorem}
\label{LiuWang-BohrTheorem}Let $\Omega$ be an irreducible bounded symmetric
domain of classical type in the sense of Hua Luokeng \cite{Hua1963}, or a
polydisc. Denote by $\left\Vert \ \right\Vert _{\Omega}$ the Minkowski norm
associated to $\Omega$. Let $f:\Omega\rightarrow\Omega$ be a holomorphic map
and let
\[
f(z)=\sum_{k=0}^{\infty}f_{k}(z)
\]
be its Taylor expansion in $k$-homogeneous polynomials $f_{k}$. Let $\phi
\in\operatorname{Aut}\Omega$ such that $\phi(f(0))=0$. Then%
\begin{equation}
\sum_{k=0}^{\infty}\frac{\left\Vert D\phi(f(0))\cdot f_{k}(Z)\right\Vert
_{\Omega}}{\left\Vert D\phi(f(0))\right\Vert _{\Omega}}<1\label{eq1}%
\end{equation}
for all $Z$ such that $\left\Vert Z\right\Vert _{\Omega}<\frac{1}{3}$.

For $\left\Vert Z\right\Vert _{\Omega}>\frac{1}{3}$, there exists a
holomorphic map $f:\Omega\rightarrow\Omega$ such that (\ref{eq1}) is not true.
\end{theorem}

Actually, the proof in \cite{LiuWang2007} depends on the following result,
which is proved by the authors for classical domains of type I and IV, using
\textit{ad hoc} computations:

\begin{theorem}
\label{differential-norm}Let $\Omega\subset V$ be a bounded circled symmetric
domain. Let $\left\Vert \ \right\Vert _{\Omega}$ be the associated spectral
norm on $V$. Let $u\in\Omega$ and let $\phi\in\operatorname{Aut}\Omega$ such
that $\phi(u)=0$. Then the operator norm of the derivative of $\phi$ at $u$
is
\[
\left\Vert \operatorname{d}\phi(u)\right\Vert _{\Omega}=\frac{1}{1-\left\Vert
u\right\Vert _{\Omega}^{2}}.
\]

\end{theorem}

In Section \ref{SEC3}, we give a classification independent proof of this
result and of Theorem \ref{LiuWang-BohrTheorem}, which is valid for any
bounded circled symmetric domain.

In the above generalization of Bohr's theorem, one considers the Taylor
expansion%
\[
f(z)=\sum_{k=0}^{\infty}f_{k}(z)
\]
of a bounded holomorphic function or map, and one asks for which $z$ the
inequality (\ref{eq1}) holds. One may ask the same type of question for other
decompositions of $f$. For example, the following problem has been considered
recently by several authors:

\begin{problem}
Let $\Omega$ be the unit ball of the $\ell^{p}$ norm:%
\[
\Omega=\left\{  \left(  z_{1},\cdots,z_{n}\right)  \in\mathbb{C}^{n}\mid
\sum_{k=1}^{n}\left\vert z_{k}\right\vert ^{p}<1\right\}  .
\]
Let $f:\Omega\rightarrow\Delta$ be a holomorphic function and consider the
expansion of $f$ in monomials%
\[
f(z)=\sum_{k_{1},\ldots,k_{n}=0}^{\infty}a_{k_{1}\cdots k_{n}}z_{1}^{k_{1}%
}\ldots z_{n}^{k_{n}}.
\]
Determine the best constant $K$ such that $z\in K\Omega$ ensures
\[
\sum_{k_{1},\ldots,k_{n}=0}^{\infty}\left\vert a_{k_{1}\cdots k_{n}}%
z_{1}^{k_{1}}\ldots z_{n}^{k_{n}}\right\vert <1
\]
for all $f:\Omega\rightarrow\Delta$.
\end{problem}

For results about this type of problem, see \cite{DineenTimoney1989},
\linebreak\cite{DineenTimoney1991}, \cite{BoasKhavinson1997},
\cite{Aizenberg2000}, \cite{Boas2000}, \linebreak%
\cite{DefantGarciaMaestre2003}

\section{H.~Bohr's theorem\label{SEC1}}%

\subsection{}

We first recall Bohr's theorem in dimension $1$. We present here a rather
elementary proof of this theorem, based on an inequality of Carath\'{e}odory
or on an improvement of this inequality using a lemma of F.~Wiener. This lemma
has been generalized in \cite{LiuWang2007} to some classes of bounded
symmetric domains. There are other proofs of Bohr's theorem in the literature,
for example by S.~Sidon \cite{Sidon1927}, following L.~Fej\'{e}r's method of
positive kernels \cite{Fejer1925}; the same proof was published later by
M.~Tomi\'{c} \cite{Tomic1962}, who added the case $a_{0}=f(0)=0$ (see Section
\ref{Bohr-Ricci}).

\begin{theorem}
[\cite{Bohr1914}]\label{HBohrTheoremDim1}Let $f:\Delta\rightarrow\Delta$ be a
holomorphic function from the unit disc $\Delta\subset\mathbb{C}$ into itself.
If $f(z)=\sum_{k=0}^{\infty}a_{k}z^{k}$, then
\begin{equation}
\sum_{k=0}^{\infty}\left\vert a_{k}z^{k}\right\vert \leq1 \label{eq2-1-1}%
\end{equation}
for $\left\vert z\right\vert \leq\frac{1}{3}$. The value $\frac{1}{3}$ is
optimal: for $\frac{1}{3}<r<1$, there exists a holomorphic function
$f:D\rightarrow D$ such that%
\[
\sum_{k=0}^{\infty}\left\vert a_{k}\right\vert r^{k}>1.
\]

\end{theorem}

For a holomorphic function $f:\Delta\rightarrow\Delta$ with Taylor expansion
$f(z)=\sum_{k=0}^{\infty}a_{k}z^{k}$, denote by $\mathfrak{M}_{f}(r)$ the sum
of absolute values:%
\[
\mathfrak{M}_{f}(r)=\sum_{k=0}^{\infty}\left\vert a_{k}\right\vert r^{k}.
\]

Let us first prove that the bound $\frac{1}{3}$ cannot be improved. For
$0<\alpha<1$, let $f$ be defined by
\[
f(z)=\frac{\alpha-z}{1-\alpha z}.
\]
Then
\[
f(z)=\alpha+\sum_{k=1}^{\infty}\left(  \alpha^{k+1}-\alpha^{k-1}\right)
z^{k},
\]
so that $f(z)=\sum_{k=0}^{\infty}a_{k}z^{k}$, with $a_{0}=\alpha$,
$a_{k}=\alpha^{k+1}-\alpha^{k-1}$. As $a_{k}$ is negative for $k>0$,
\[
\mathfrak{M}(r)=\alpha-\sum_{k=1}^{\infty}\left(  \alpha^{k+1}-\alpha
^{k-1}\right)  r^{k}=2\alpha-f(r).
\]
Then $\mathfrak{M}(r)>1$ when $r>\frac{1}{1+2\alpha}$; as $\frac{1}{1+2\alpha
}\rightarrow\frac{1}{3}+0$ when $\alpha\rightarrow1-0$, there exists for each
$r>\frac{1}{3}$ a holomorphic function $f:\Delta\rightarrow\Delta$ such that
$\mathfrak{M}_{f}(r)>1$.

To prove the direct part of Theorem \ref{HBohrTheoremDim1}, one may use one of
the following lemmas.

\begin{lemma}
Let $f:\Delta\rightarrow\Delta$ be a holomorphic function, with $f(z)=\sum
_{k=0}^{\infty}a_{k}z^{k}$. Then for all $k>0$,
\begin{equation}
\left\vert a_{k}\right\vert \leq2(1-\left\vert a_{0}\right\vert ).
\label{eq-dim1-1}%
\end{equation}

\end{lemma}

\begin{lemma}
\label{Wiener}Let $f:\Delta\rightarrow\Delta$ be a holomorphic function, with
$f(z)=\sum_{k=0}^{\infty}a_{k}z^{k}$. Then for all $k>0$,
\begin{equation}
\left\vert a_{k}\right\vert \leq1-\left\vert a_{0}\right\vert ^{2}.
\label{eq-dim1-2}%
\end{equation}

\end{lemma}

The inequalities (\ref{eq-dim1-1}) result directly from inequalities, due to
C.~Carath\'{e}odory, for holomorphic functions in $\Delta$ with positive real
part; they can also be used to prove Bohr's theorem. Note that
(\ref{eq-dim1-2}) is slightly sharper than (\ref{eq-dim1-1}), as
$(1-t^{2})\leq2(1-t)$ for all $t\in\mathbb{R}$. We will prove directly the
second lemma.

First we prove the special case:%
\begin{equation}
\left\vert a_{1}\right\vert \leq1-\left\vert a_{0}\right\vert ^{2}.
\label{eq-dim1-4}%
\end{equation}
Consider the automorphism $\phi$ of $\Delta$ defined by
\[
\phi(z)=\frac{z-a_{0}}{1-\overline{a_{0}}z};
\]
we have
\begin{align*}
\phi(a_{0})  &  =0,\qquad\phi^{\prime}(a_{0})=\frac{1}{1-a_{0}\overline{a_{0}%
}},\\
\left(  \phi\circ f\right)  ^{\prime}(0)  &  =\phi^{\prime}(a_{0})f^{\prime
}(0)=\frac{a_{1}}{1-a_{0}\overline{a_{0}}}.
\end{align*}
As $\phi\circ f$ maps $\Delta$ into $\Delta$ and $(\phi\circ f)(0)=0$, we have
$\left\vert \left(  \phi\circ f\right)  ^{\prime}(0)\right\vert \leq1$, which
implies%
\[
\left\vert a_{1}\right\vert \leq1-a_{0}\overline{a_{0}}.
\]

From (\ref{eq-dim1-4}), we will now deduce (\ref{eq-dim1-1}) for all $k>0$.
For $k>0$, define%
\[
g_{k}(z)=\frac{1}{k}\sum_{j=1}^{k}f(\operatorname{e}^{2\operatorname{i}\pi
j/k}z)=\sum_{m=0}^{\infty}a_{km}z^{km}.
\]
From the first equality, it results that $g_{k}$ maps $\Delta$ into $\Delta$,
and so does also $h_{k}$ defined by%
\[
h_{k}(z)=\sum_{m=0}^{\infty}a_{km}z^{m}.
\]
We have $h_{k}(0)=a_{0}$ and $h_{k}^{\prime}(0)=a_{k}$ ; applying the
inequality (\ref{eq-dim1-4}) to the function $h_{k}$ yields
\[
\left\vert a_{k}\right\vert \leq1-a_{0}\overline{a_{0}}.
\]

\begin{proof}
[Proof of Bohr's theorem]Using (\ref{eq-dim1-2}), we deduce
\[
\mathfrak{M}(r)=\sum_{k=0}^{\infty}\left\vert a_{k}\right\vert r^{k}%
\leq\left\vert a_{0}\right\vert +\left(  1-a_{0}\overline{a_{0}}\right)
\frac{r}{1-r}%
\]
and, for $r\leq\frac{1}{3}$,
\[
\mathfrak{M}(r)\leq\left\vert a_{0}\right\vert +\frac{1}{2}\left(
1-a_{0}\overline{a_{0}}\right)  =1-\frac{1}{2}\left(  1-\left\vert
a_{0}\right\vert ^{2}\right)  <1,
\]
as $\left\vert a_{0}\right\vert =\left\vert f(0)\right\vert <1$.
\end{proof}

\subsection{ \label{Bohr-Ricci}}

Another Bohr type theorem holds in dimension $1$ for the class of holomorphic
functions $f:\Delta\rightarrow\Delta$ such that $f(0)=0$. It was proved in
\cite{Tomic1962} using an analogue of the argument in \cite{Sidon1927}. An
elementary proof for this, which uses another result of L.\ Fej\'{e}r
\cite{Fejer1914}, can also be found in \cite{Landau1925}.

This type of problem was generalized and partially solved in \cite{Ricci1955}
and \cite{Bombieri1962} by G.\ Ricci and E.\ Bombieri.

\begin{definition}
Consider the following classes of analytic functions:%
\begin{align*}
\mathcal{F}_{0} &  =\left\{  f:\Delta\rightarrow\Delta\right\}  ,\\
\mathcal{F}_{0,\alpha} &  =\left\{  f:\Delta\rightarrow\Delta\mid
f(0)=\alpha\right\}  \qquad(0\leq\alpha<1),\\
\mathcal{F}_{m} &  =\left\{  f:\Delta\rightarrow\Delta\mid f(z)=\sum
_{k=m}^{\infty}a_{k}z^{k}\right\}  \qquad(m\in\mathbb{N)},\\
\mathcal{F}_{m,\alpha} &  =\left\{  f:\Delta\rightarrow\Delta\mid f(z)=\alpha
z^{m}+\sum_{k=m+1}^{\infty}a_{k}z^{k}\right\}  \qquad(m\in\mathbb{N}%
,\quad0\leq\alpha<1).
\end{align*}
Define the corresponding \emph{Bohr numbers} by%
\begin{align*}
B_{0} &  =\sup\left\{  r\mid\mathfrak{M}_{f}(r)<1\mathrm{\ for\ all\ }%
f\in\mathcal{F}_{0}\right\}  ,\\
B_{0,\alpha} &  =\sup\left\{  r\mid\mathfrak{M}_{f}(r)<1\mathrm{\ for\ all\ }%
f\in\mathcal{F}_{0,\alpha}\right\}  ,\\
B_{m} &  =\sup\left\{  r\mid\mathfrak{M}_{f}(r)<1\mathrm{\ for\ all\ }%
f\in\mathcal{F}_{m}\right\}  ,\\
B_{m,\alpha} &  =\sup\left\{  r\mid\mathfrak{M}_{f}(r)<1\mathrm{\ for\ all\ }%
f\in\mathcal{F}_{m,\alpha}\right\}  .
\end{align*}

\end{definition}

The original Bohr theorem means that $B_{0}=\frac{1}{3}$.

The result in \cite{Tomic1962} (already proved in \cite{Landau1925}) says that
$B_{1}>\frac{1}{2}$. In  \cite{Ricci1955}, the author proves, among other
results, that  $\frac{3}{5}<B_{1}\leq\frac{1}{\sqrt{2}}$. The optimal result
$B_{1}=\frac{1}{\sqrt{2}}$ was obtained by \cite{Bombieri1962}.

Let us prove that Ricci's result $B_{1}>\frac{3}{5}$ easily follows from
Wiener's lemma \ref{Wiener}.

\begin{proposition}
\label{Bohr_Ricci_0}Let $f:\Delta\rightarrow\Delta$ be a holomorphic function
such that $f(0)=0,$ $f(z)=\sum_{k=1}^{\infty}a_{k}z^{k}$. Then
\[
\sum_{k=0}^{\infty}\left\vert a_{k}z^{k}\right\vert \leq1
\]
for $\left\vert z\right\vert \leq\frac{3}{5}$.
\end{proposition}

\begin{proof}
Consider the function%
\[
g(z)=\frac{f(z)}{z}=\sum_{k=0}^{\infty}a_{k+1}z^{k}.
\]
By the Schwarz lemma, $g$ maps $\Delta$ into $\Delta.$ Lemma \ref{Wiener}
applied to $g$ gives
\[
\left\vert a_{k}\right\vert \leq1-\left\vert a_{1}\right\vert ^{2}\qquad(k>1).
\]
Hence%
\begin{align*}
\mathfrak{M}_{f}(r)  & =\sum_{k=0}^{\infty}\left\vert a_{k}\right\vert
r^{k}\leq\left\vert a_{1}\right\vert r+\left(  1-\left\vert a_{1}\right\vert
^{2}\right)  \frac{r^{2}}{1-r},\\
\mathfrak{M}_{f}\left(  \frac{3}{5}\right)    & \leq\frac{3}{5}\left(
\left\vert a_{1}\right\vert +\frac{3}{2}\left(  1-\left\vert a_{1}\right\vert
^{2}\right)  \right)  =1-\frac{9}{10}\left(  \frac{1}{3}-\left\vert
a_{1}\right\vert \right)  ^{2}\leq1.
\end{align*}

\end{proof}

\section{H.~Bohr's theorem for bounded symmetric domains\label{SEC3}}

\subsection{}

Let $\Omega\subset V$ be a bounded circled homogeneous domain in a finite
dimensional complex vector space $V$. See the appendix for notations and
general results. Denote by $\left\Vert \ \right\Vert _{\Omega}$ the spectral
norm associated to $\Omega$.

The main result in \cite{LiuWang2007} is the following theorem, which is
proved there for domains of type $I$ (rectangular matrices) and for polydiscs.

\begin{theorem}
\label{HBohrThm-BSD}Let $\Omega$ be a bounded circled symmetric domain. Denote
by $\left\Vert \ \right\Vert _{\Omega}$ the spectral norm associated to
$\Omega$. Let $f:\Omega\rightarrow\Omega$ be a holomorphic map and let
\[
f=\sum_{k=0}^{\infty}f_{k}%
\]
be its Taylor expansion in $k$-homogeneous polynomials $f_{k}$. Let $\phi
\in\operatorname{Aut}\Omega$ such that $\phi(f(0))=0$. Then%
\begin{equation}
\sum_{k=0}^{\infty}\frac{\left\Vert \operatorname{d}\phi(f(0))\cdot
f_{k}(z)\right\Vert _{\Omega}}{\left\Vert \operatorname{d}\phi
(f(0))\right\Vert _{\Omega}}<1 \label{eq2-07}%
\end{equation}
for all $z$ such that $\left\Vert z\right\Vert _{\Omega}\leq\frac{1}{3}$.

The bound $\frac{1}{3}$ is optimal: for each $z\in\Omega$ with $\left\Vert
z\right\Vert _{\Omega}>\frac{1}{3}$, there exists a holomorphic map
$f=\sum_{k=0}^{\infty}f_{k}:\Omega\rightarrow\Omega$ such that
\[
\sum_{k=0}^{\infty}\frac{\left\Vert \operatorname{d}\phi(f(0))\cdot
f_{k}(z)\right\Vert _{\Omega}}{\left\Vert \operatorname{d}\phi
(f(0))\right\Vert _{\Omega}}>1.
\]

\end{theorem}

This theorem will be proved below for all bounded symmetric domains.

\subsection{}

The following properties generalize to all bounded circled symmetric domains
results in \cite{LiuWang2007} for the differential at $u\in\Omega$ of an
automorphism $\phi\in\operatorname{Aut}\Omega$ such that $\phi(u)=0,$ and may
be of independent interest.

For $u\in V$, denote by $\tau_{u}$ the translation $z\mapsto z+u$ and by
$\widetilde{\tau}_{u}$ the rational map%
\[
\widetilde{\tau}_{u}(z)=z^{u},
\]
where $z^{u}$ is the \emph{quasi-inverse}
\[
z^{u}=B(z,u)^{-1}\left(  z-Q(z)u\right)  ,
\]
defined in the open set of points $z$ such that $B(z,u)=\operatorname{id}%
_{V}-D(z,u)+Q(x)Q(u)$ is invertible.

Recall that for $u\in\Omega$, the operator $B(u,u)$ is positive (with respect
to the Hermitian scalar product on $V$: $\left(  x\mid y\right)
=\operatorname{tr}D(x,y)$), so that $B(u,u)^{t}$ is well defined for
$t\in\mathbb{R}$. Let $u\in\Omega$ and consider a spectral decomposition%
\begin{align*}
u  &  =\lambda_{1}e_{1}+\cdots+\lambda_{r}e_{r},\\
1  &  >\lambda_{1}\geq\cdots\geq\lambda_{r}\geq0,
\end{align*}
where $\mathbf{e}=\left(  e_{1},\ldots,e_{r}\right)  $ is a frame. Then, by
(A11),
\[
B(u,u)=\sum_{0\leq i\leq j\leq r}\left(  1-\lambda_{i}^{2}\right)  \left(
1-\lambda_{j}^{2}\right)  p_{ij}%
\]
and $B(u,u)^{t}$ is given by
\begin{equation}
B(u,u)^{t}=\sum_{0\leq i\leq j\leq r}\left(  1-\lambda_{i}^{2}\right)
^{t}\left(  1-\lambda_{j}^{2}\right)  ^{t}p_{ij}, \label{eq2-0F}%
\end{equation}
where $(p_{ij})$ is the family of orthogonal projectors onto the subspaces of
the simultaneous Peirce decomposition with respect to the frame $\mathbf{e}%
=\left(  e_{1},\ldots,e_{r}\right)  $.

The following result is well known.

\begin{lemma}
\label{L1}For $u\in\Omega$, the map
\begin{equation}
\phi_{u}=\widetilde{\tau}_{u}\circ B(u,u)^{-1/2}\circ\tau_{-u} \label{eq2-01}%
\end{equation}
is an automorphism of $\Omega$ which sends $u$ to $0$. The derivative of
$\phi_{u}$ at $u$ is%
\begin{equation}
\operatorname{d}\phi_{u}(u)=B(u,u)^{-1/2}. \label{eq2-02}%
\end{equation}

\end{lemma}

\begin{proof}
Let
\[
\psi_{u}=\tau_{u}\circ B(u,u)^{1/2}\circ\widetilde{\tau}_{-u}.
\]
Then, by \cite[Proposition 9.8 (1)]{Loos1977}, $\psi_{u}$ is an automorphism
of $\Omega$. As $\widetilde{\tau}_{-u}(0)=0$ and $B(u,u)^{1/2}$ is linear, we
have $\psi_{u}(0)=u$. The inverse of $\psi_{u}$ is
\[
\phi_{u}=\widetilde{\tau}_{u}\circ B(u,u)^{-1/2}\circ\tau_{-u}%
\]
and $\phi_{u}(u)=0$. The derivative of $\widetilde{\tau}_{u}$ is
\[
\operatorname{d}\widetilde{\tau}_{u}(z)=B(z,u)^{-1}%
\]
(see \cite[Proposition III.4.1 (ii)]{Roos2000}). As $B(0,u)=\operatorname{id}%
_{V}$, it follows that
\[
\operatorname{d}\phi_{u}(u)=B(u,u)^{-1/2}.
\]

\end{proof}

Denote by $\left\Vert ~\right\Vert $ the Hermitian norm on $V$ ($\left\Vert
z\right\Vert ^{2}=\operatorname{tr}D(z,z)$) and by $\left\Vert ~\right\Vert $
the associated operator norm for linear endomorphisms of $V$. The
\emph{spectral norm} $\left\Vert ~\right\Vert _{\Omega}$ on $V$ is given by
\[
\left\Vert z\right\Vert _{\Omega}^{2}=\frac{1}{2}\left\Vert D(z,z)\right\Vert
=\left\Vert Q(z)\right\Vert ;
\]
the unit ball of this norm is $\Omega$ (see the appendix). We denote also by
$\left\Vert ~\right\Vert _{\Omega}$ the associated operator norm for linear
endomorphisms of $V$.

\begin{lemma}
\label{L4}For $u\in\Omega$, $t\in\mathbb{R}$,%
\begin{equation}
\left\Vert B(u,u)^{t}\right\Vert _{\Omega}=\left\Vert B(u,u)^{t}\right\Vert .
\label{eq2-0D}%
\end{equation}

\end{lemma}

That is, the operator $B(u,u)^{t}$ has the same operator norm with respect to
the spectral norm or to the Hermitian norm.

\begin{proof}
Let $u\in\Omega$ and consider a spectral decomposition%
\begin{align*}
u  &  =\lambda_{1}e_{1}+\cdots+\lambda_{r}e_{r},\\
1  &  >\lambda_{1}\geq\cdots\geq\lambda_{r}\geq0,
\end{align*}
where $\mathbf{e}=\left(  e_{1},\ldots,e_{r}\right)  $ is a frame. Then, by
(A11),
\[
B(u,u)=\sum_{0\leq i\leq j\leq r}\left(  1-\lambda_{i}^{2}\right)  \left(
1-\lambda_{j}^{2}\right)  p_{ij}%
\]
and
\[
B(u,u)^{t}=\sum_{0\leq i\leq j\leq r}\left(  1-\lambda_{i}^{2}\right)
^{t}\left(  1-\lambda_{j}^{2}\right)  ^{t}p_{ij},
\]
where $(p_{ij})$ is the family of orthogonal projectors onto the subspaces of
the simultaneous Peirce decomposition with respect to the frame $\mathbf{e}%
=\left(  e_{1},\ldots,e_{r}\right)  $.

For $t\geq0$, let $\mu_{j}\geq0$ be defined by%
\[
\left(  1-\lambda_{i}^{2}\right)  ^{t}=1-\mu_{j}^{2}%
\]
and define
\[
w=\mu_{1}e_{1}+\cdots+\mu_{r}e_{r}.
\]
Then, again by (A11),
\[
B(u,u)^{t}=B(w,w).
\]
Using (A5), we get for $t\geq0$ and all $z\in V$%
\begin{equation}
Q(B(u,u)^{t}z)=B(u,u)^{t}Q(z)B(u,u)^{t}. \label{eq2-0A}%
\end{equation}
This implies
\begin{equation}
Q(B(u,u)^{-t}z)=B(u,u)^{-t}Q(z)B(u,u)^{-t}, \label{eq2-0B}%
\end{equation}
so that (\ref{eq2-0A}) is finally valid for all $t\in\mathbb{R}$.

Now, by (A12), for $t\in\mathbb{R}$, $z\in V$,
\[
\left\Vert B(u,u)^{t}z\right\Vert _{\Omega}^{2}=\left\Vert Q(B(u,u)^{t}%
z)\right\Vert \leq\left\Vert B(u,u)^{t}\right\Vert ^{2}\left\Vert
Q(z)\right\Vert =\left\Vert B(u,u)^{t}\right\Vert ^{2}\left\Vert z\right\Vert
_{\Omega}^{2}.
\]
This implies%
\[
\left\Vert B(u,u)^{t}\right\Vert _{\Omega}\leq\left\Vert B(u,u)^{t}\right\Vert
.
\]
Recall that
\[
B(u,u)^{t}=\sum_{0\leq i\leq j\leq r}\left(  1-\lambda_{i}^{2}\right)
^{t}\left(  1-\lambda_{j}^{2}\right)  ^{t}p_{ij},
\]
where the $p_{ij}$'s are orthogonal projections with respect to the Hermitian
metric. Then $\left\Vert B(u,u)^{t}\right\Vert $ is the greatest eigenvalue of
$B(u,u)^{t}$, and $\left\Vert B(u,u)^{t}\right\Vert _{\Omega}\geq\left\Vert
B(u,u)^{t}\right\Vert .$
\end{proof}

For $u\in\Omega$ with spectral decomposition $u=\lambda_{1}e_{1}%
+\cdots+\lambda_{r}e_{r}$, $1>\lambda_{1}\geq\cdots\geq\lambda_{r}\geq0$,
denote by $\beta(u)$ the greatest eigenvalue of $B(u,u).$

\begin{lemma}
\label{L6} Let $\Omega$ be a bounded \emph{irreducible} symmetric domain with
invariants $a,b,r$. For $u\in\Omega$ with spectral decomposition
$u=\lambda_{1}e_{1}+\cdots+\lambda_{r}e_{r}$, $1>\lambda_{1}\geq\cdots
\geq\lambda_{r}\geq0$, the value of $\beta(u)$ is

\begin{enumerate}
\item if $\Omega$ is the unit disc of $\mathbb{C}$ ($b=0$, $r=1$),
\[
\beta(u)=\left(  1-\left\Vert u\right\Vert _{\Omega}^{2}\right)  ^{2};
\]

\item if $\Omega$ is the Hermitian ball of dimension $n>1$ ($b=n-1$, $r=1$),%
\[
\beta(u)=1-\left\Vert u\right\Vert _{\Omega}^{2};
\]

\item if $\Omega$ is of tube type ($b=0$),
\[
\beta(u)=\left(  1-\lambda_{r}^{2}\right)  ^{2};
\]

\item if $\Omega$ is of non tube type ($b>0$),
\[
\beta(u)=1-\lambda_{r}^{2}.
\]

\end{enumerate}
\end{lemma}

\begin{proof}
The lemma follows from the fact that the eigenvalues of $B(u,u)$ are:
$(1-\lambda_{i}^{2})^{2}$ ($1\leq i\leq r$), $(1-\lambda_{i}^{2}%
)(1-\lambda_{j}^{2})$ ($1\leq i<j\leq r$) and \emph{only in the non tube case}
$1-\lambda_{i}^{2}$ ($1\leq i\leq r$).
\end{proof}

Note that $\lambda_{1}=\lambda_{r}$ occurs if and only if $u$ is a scalar
multiple of a \emph{maximal tripotent} (an element of the Shilov boundary of
$\Omega$). In this case, one has $\beta(u)=\left(  1-\left\Vert u\right\Vert
_{\Omega}^{2}\right)  ^{2}$ or $\beta(u)=1-\left\Vert u\right\Vert _{\Omega
}^{2}$, depending on whether the domain $\Omega$ is of tube type or not.

\begin{proposition}
\label{L3}Let $\Omega$ be a bounded circled symmetric domain. For $u\in\Omega$
we have
\begin{align}
\left\Vert B(u,u)^{-1/2}\right\Vert _{\Omega}  &  =\frac{1}{1-\left\Vert
u\right\Vert _{\Omega}^{2}},\label{eq2-04}\\
\left\Vert B(u,u)^{-1/2}u\right\Vert _{\Omega}  &  =\frac{\left\Vert
u\right\Vert _{\Omega}}{1-\left\Vert u\right\Vert _{\Omega}^{2}}%
,\label{eq2-05}\\
\left\Vert B(u,u)^{1/2}\right\Vert _{\Omega}  &  =\beta(u)^{1/2}.
\label{eq2-09}%
\end{align}
For any automorphism $\phi\in\operatorname{Aut}\Omega$ such that $\phi(u)=0$,
\begin{align}
\left\Vert \operatorname{d}\phi(u)\right\Vert _{\Omega}  &  =\frac
{1}{1-\left\Vert u\right\Vert _{\Omega}^{2}},\label{eq2-03}\\
\left\Vert \operatorname{d}\phi(u)u\right\Vert _{\Omega}  &  =\frac{\left\Vert
u\right\Vert _{\Omega}}{1-\left\Vert u\right\Vert _{\Omega}^{2}}%
\label{eq2-0E}\\
\left\Vert \left(  \operatorname{d}\phi(u)\right)  ^{-1}\right\Vert _{\Omega}
&  =\beta(u)^{1/2}. \label{eq2-06}%
\end{align}

\end{proposition}

\begin{proof}
Let $u=\lambda_{1}e_{1}+\cdots+\lambda_{r}e_{r}$, $1>\lambda_{1}\geq\cdots
\geq\lambda_{r}\geq0$ be a spectral decomposition of $u$. The operator norm
$\left\Vert B(u,u)^{-1/2}\right\Vert $ with respect to the Hermitian product
is its greatest eigenvalue $\left(  1-\lambda_{1}^{2}\right)  ^{-1}=\left(
1-\left\Vert u\right\Vert _{\Omega}^{2}\right)  ^{-1}$; then Lemma \ref{L4}
implies (\ref{eq2-04}). In the same way, Lemma \ref{L4} implies (\ref{eq2-09}%
). We have%
\[
B(u,u)^{-1/2}u=\sum_{i=1}^{r}\frac{\lambda_{i}}{\left(  1-\lambda_{i}%
^{2}\right)  ^{1/2}}e_{i},
\]
hence
\[
\left\Vert B(u,u)^{-1/2}u\right\Vert _{\Omega}=\frac{\lambda_{1}}{\left(
1-\lambda_{1}^{2}\right)  ^{1/2}}%
\]
and (\ref{eq2-05}).

Let $\phi$ be any automorphism $\phi\in\operatorname{Aut}\Omega$ such that
$\phi(u)=0$. Then $\phi=k\circ\phi_{u}$, where $k$ is a linear automorphism of
$\Omega$; the spectral norm is invariant by $k$ and
\begin{align*}
\left\Vert \operatorname{d}\phi(u)\right\Vert _{\Omega}  &  =\left\Vert k\circ
B(u,u)^{-1/2}\right\Vert _{\Omega}=\left\Vert B(u,u)^{-1/2}\right\Vert
_{\Omega},\\
\left\Vert \operatorname{d}\phi(u)u\right\Vert _{\Omega}  &  =\left\Vert
k\circ B(u,u)^{-1/2}u\right\Vert _{\Omega}=\left\Vert B(u,u)^{-1/2}%
u\right\Vert _{\Omega},\\
\left\Vert \left(  \operatorname{d}\phi(u)\right)  ^{-1}\right\Vert _{\Omega}
&  =\left\Vert B(u,u)^{1/2}\right\Vert _{\Omega}.
\end{align*}

\end{proof}

\subsection{}

Let $\Omega$ be a bounded circled symmetric domain. If $f:\Omega
\rightarrow\Omega$ is a holomorphic map, denote by
\[
f(z)=\sum_{k=0}^{\infty}f_{k}(z)
\]
its Taylor expansion in $k$-homogeneous polynomials $f_{k}$. The following
lemma generalizes Lemma \ref{Wiener} to bounded symmetric domains.

\begin{lemma}
\label{L5}Let $u=f(0)$ and let $\phi\in\operatorname{Aut}\Omega$ such that
$\phi(u)=0$. Then%
\begin{equation}
\left\Vert \operatorname{d}\phi(u)\cdot f_{k}(z)\right\Vert _{\Omega}%
\leq\left\Vert z\right\Vert _{\Omega}^{k}. \label{eq2-0C}%
\end{equation}

\end{lemma}

\begin{proof}
For $k>1$, consider $g_{k}:\Omega\rightarrow\Omega$ defined by
\[
g_{k}(z)=\frac{1}{k}\sum_{j=1}^{k}f\left(  \operatorname{e}^{2\operatorname{i}%
\pi j/k}z\right)  =\sum_{m=0}^{\infty}f_{mk}(z).
\]
Let $h_{k}=\phi\circ g_{k}$. Then
\begin{align*}
h_{k}(z)  &  =\phi\left(  \sum_{m=0}^{\infty}f_{mk}(z)\right) \\
&  =\phi(u)+\operatorname{d}\phi(u)\cdot\left(  \sum_{m=1}^{\infty}%
f_{mk}(z)\right)  +\cdots=\operatorname{d}\phi(u)\cdot f_{k}(z)+\cdots
\end{align*}
and, for $z\in\Omega$,
\[
\operatorname{d}\phi(u)\cdot f_{k}(z)=\int_{0}^{1}\phi\circ g_{k}\left(
\operatorname{e}^{2\operatorname{i}\pi t}z\right)  \operatorname{e}%
^{-2\operatorname{i}\pi kt}\operatorname{d}t.
\]
This implies
\[
\left\Vert \operatorname{d}\phi(u)\cdot f_{k}(z)\right\Vert _{\Omega}\leq1
\]
for $z\in\Omega$, and by homogeneity of $f_{k}$,
\[
\left\Vert \operatorname{d}\phi(u)\cdot f_{k}(z)\right\Vert _{\Omega}%
\leq\left\Vert z\right\Vert _{\Omega}^{k}%
\]
for $z\in V$.
\end{proof}

\begin{proposition}
\label{P1}Let $\Omega$ be a bounded circled symmetric domain. Let $u\in\Omega$
and let $\phi\in\operatorname{Aut}\Omega$ such that $\phi(u)=0$.

Then%
\begin{equation}
\sum_{k=0}^{\infty}\frac{\left\Vert \operatorname{d}\phi(u)\cdot
f_{k}(z)\right\Vert _{\Omega}}{\left\Vert \operatorname{d}\phi(u)\right\Vert
_{\Omega}}<1 \label{eq2-10}%
\end{equation}
for all $f:\Omega\rightarrow\Omega$ such that $f(0)=u$ and for all $z$ such
that $\left\Vert z\right\Vert _{\Omega}\leq\frac{1}{2+\left\Vert u\right\Vert
_{\Omega}}$.
\end{proposition}

\begin{proof}
The proof goes along the same lines as in \cite{LiuWang2007}. For $z\in V$ we
have by Lemma \ref{L5}%
\[
\sum_{k=0}^{\infty}\frac{\left\Vert \operatorname{d}\phi(f(0))\cdot
f_{k}(z)\right\Vert _{\Omega}}{\left\Vert \operatorname{d}\phi
(f(0))\right\Vert _{\Omega}}\leq\left\Vert u\right\Vert _{\Omega}+\frac
{1}{\left\Vert \operatorname{d}\phi(u)\right\Vert _{\Omega}}\sum_{k=1}%
^{\infty}\left\Vert z\right\Vert _{\Omega}^{k}.
\]
Using $\left\Vert \operatorname{d}\phi(u)\right\Vert _{\Omega}=\frac
{1}{1-\left\Vert u\right\Vert _{\Omega}^{2}}$ from Proposition \ref{L3}, we
get for $\left\Vert z\right\Vert _{\Omega}\leq r<1$%
\[
\sum_{k=0}^{\infty}\frac{\left\Vert \operatorname{d}\phi(f(0))\cdot
f_{k}(z)\right\Vert _{\Omega}}{\left\Vert \operatorname{d}\phi
(f(0))\right\Vert _{\Omega}}\leq\left\Vert u\right\Vert _{\Omega}+\left(
1-\left\Vert u\right\Vert _{\Omega}^{2}\right)  \frac{r}{1-r}.
\]
For fixed $u$, the right hand side is less than $1$ if and only if $r<\frac
{1}{2+\left\Vert u\right\Vert _{\Omega}}$.
\end{proof}

\begin{proposition}
\label{P2}Let $\Omega$ be a bounded circled symmetric domain. Let $u\in\Omega$
and let $\phi\in\operatorname{Aut}\Omega$ such that $\phi(u)=0$. Assume
$\left\Vert u\right\Vert _{\Omega}>\frac{1}{3}$ and $\frac{1}{1+2\left\Vert
u\right\Vert _{\Omega}}<a<1$.

Then there exist a holomorphic map $f:\Omega\rightarrow\Omega$ with $f(0)=u$
and $z\in\Omega$ with $\left\Vert z\right\Vert _{\Omega}=a$, such that
\[
\sum_{k=0}^{\infty}\frac{\left\Vert \operatorname{d}\phi(u)\cdot
f_{k}(z)\right\Vert _{\Omega}}{\left\Vert \operatorname{d}\phi(u)\right\Vert
_{\Omega}}>1.
\]

\end{proposition}

\begin{proof}
Consider $u\in\Omega$ with spectral decomposition
\begin{align*}
u  &  =\lambda_{1}e_{1}+\cdots+\lambda_{r}e_{r},\\
1  &  >\left\Vert u\right\Vert _{\Omega}=\lambda_{1}\geq\cdots\geq\lambda
_{r}\geq0.
\end{align*}
Define $f:\Omega\rightarrow\Omega$ by%
\[
f(z)=\sum_{i=1}^{r}\frac{\lambda_{i}-\left(  z\mid e_{i}\right)  }%
{1-\lambda_{i}\left(  z\mid e_{i}\right)  }e_{i}.
\]
This is a well defined holomorphic map, as $\left\vert \left(  z\mid
e_{i}\right)  \right\vert <1$ for all $z\in\Omega$, due to the convexity of
$\Omega$. From
\[
\frac{\lambda-\zeta}{1-\lambda\zeta}=\lambda+\sum_{k=1}^{\infty}\left(
\lambda^{k+1}-\lambda^{k-1}\right)  \zeta^{k},
\]
we get the Taylor expansion of $f$:%
\begin{align*}
f  &  =\sum_{k=0}^{\infty}f_{k},\\
u  &  =f(0)=f_{0}(z),\\
f_{k}(z)  &  =\sum_{i=1}^{r}\left(  \lambda_{i}^{k+1}-\lambda_{i}%
^{k-1}\right)  \left(  z\mid e_{i}\right)  ^{k}e_{i}\qquad(k\geq1).
\end{align*}
In particular, taking $z_{0}=ae_{1}$, we have%
\[
f_{k}(z_{0})=\left(  \lambda_{1}^{k+1}-\lambda_{1}^{k-1}\right)  a^{k}%
e_{1}\qquad(k\geq1).
\]
Take $\phi_{u}$ as in Lemma \ref{L1}. Then
\begin{align*}
\operatorname{d}\phi_{u}(u)  &  =B(u,u)^{-1/2},\\
\left\Vert \operatorname{d}\phi_{u}(u)\right\Vert _{\Omega}  &  =\frac
{1}{1-\left\Vert u\right\Vert _{\Omega}^{2}}.
\end{align*}
From (\ref{eq2-0F}) we get%
\begin{align*}
B(u,u)^{-1/2}e_{1}  &  =(1-\left\Vert u\right\Vert _{\Omega}^{2})^{-1}e_{1},\\
\left\Vert \operatorname{d}\phi_{u}(u)\cdot u\right\Vert  &  =(1-\left\Vert
u\right\Vert _{\Omega}^{2})^{-1}\left\Vert u\right\Vert _{\Omega},\\
\operatorname{d}\phi_{u}(u)\cdot f_{k}(z_{0})  &  =-a^{k}\lambda_{1}%
^{k-1}e_{1},\\
\left\Vert \operatorname{d}\phi_{u}(u)\cdot f_{k}(z_{0})\right\Vert _{\Omega}
&  =a^{k}\lambda_{1}^{k-1}=a^{k}\left\Vert u\right\Vert _{\Omega}^{k-1}.
\end{align*}
Hence%
\begin{align*}
\sum_{k=0}^{\infty}\frac{\left\Vert \operatorname{d}\phi(u)\cdot f_{k}%
(z_{0})\right\Vert _{\Omega}}{\left\Vert \operatorname{d}\phi(u)\right\Vert
_{\Omega}}  &  =\left\Vert u\right\Vert _{\Omega}+(1-\left\Vert u\right\Vert
_{\Omega}^{2})\sum_{k=1}^{\infty}a^{k}\left\Vert u\right\Vert _{\Omega}%
^{k-1}\\
&  =\left\Vert u\right\Vert _{\Omega}+(1-\left\Vert u\right\Vert _{\Omega}%
^{2})\frac{a}{1-a\left\Vert u\right\Vert _{\Omega}}.
\end{align*}
An elementary computation shows that this is greater than $1$ if and only if
\[
a>\frac{1}{1+2\left\Vert u\right\Vert _{\Omega}}.
\]

\end{proof}

\begin{proof}
[Proof of Theorem \ref{HBohrThm-BSD}]Proposition \ref{P1} shows that
(\ref{eq2-10}) is satisfied for all maps $f:\Omega\rightarrow\Omega$ and all
$z$ such that $\left\Vert z\right\Vert _{\Omega}<\frac{1}{3}$. Moreover, for
$\left\Vert z\right\Vert _{\Omega}=\frac{1}{3}$,we have%
\begin{align*}
\sum_{k=0}^{\infty}\frac{\left\Vert \operatorname{d}\phi(f(0))\cdot
f_{k}(z)\right\Vert _{\Omega}}{\left\Vert \operatorname{d}\phi
(f(0))\right\Vert _{\Omega}}  &  \leq\left\Vert f(0)\right\Vert _{\Omega
}+\left(  1-\left\Vert f(0)\right\Vert _{\Omega}^{2}\right)  \sum
_{k=1}^{\infty}\frac{1}{3^{k}}\\
&  =\left\Vert f(0)\right\Vert _{\Omega}+\frac{1-\left\Vert f(0)\right\Vert
_{\Omega}^{2}}{2}=1-\frac{\left(  1-\left\Vert f(0)\right\Vert _{\Omega
}\right)  ^{2}}{2}<1.
\end{align*}

As $\frac{1}{1+2\left\Vert u\right\Vert _{\Omega}}\rightarrow\frac{1}{3}$ as
$\left\Vert u\right\Vert _{\Omega}\rightarrow1-0$, proposition \ref{P2}
implies that $\frac{1}{3}$ is the optimal bound.
\end{proof}

\subsection{Open questions}

\begin{problem}
Let $f:\Omega\rightarrow\Omega$ be a holomorphic map.

\begin{enumerate}
\item With the assumption $f(0)=0$, Theorem \ref{HBohrThm-BSD} gives
\begin{equation}
\sum_{k=0}^{\infty}\left\Vert f_{k}(z)\right\Vert _{\Omega}<1\label{eq2-08}%
\end{equation}
for all $z$ such that $\left\Vert z\right\Vert _{\Omega}<\frac{1}{2}$. Is the
optimal bound equal to $\frac{1}{\sqrt{2}}$, as proved by E.\ Bombieri
\cite{Bombieri1962} in the one dimensional case?

\item What is the same optimal bound for all maps $f$ satisfying $f(0)=u$,
with $u\in\Omega$ fixed? Propositions \ref{P1} and \ref{P2} show that this
optimal bound belongs to $\left[  \frac{1}{2+\left\Vert u\right\Vert _{\Omega
}},\frac{1}{1+2\left\Vert u\right\Vert _{\Omega}}\right]  $. In the one
dimensional case, this is Ricci's estimate (see \cite{Ricci1955}). But
Bombieri's results for the one dimensional case show that this estimate may be
sharpened: in particular, the optimal bound is $\frac{1}{\sqrt{2}}$ when $u=0$,
$\frac{1}{1+2\left\Vert u\right\Vert _{\Omega}}$ when $\frac{1}{2}<\left\Vert
u\right\Vert _{\Omega}<1$.
\end{enumerate}
\end{problem}

For $\Omega$ a bounded circled symmetric domain, a natural problem
generalizing H.\ Bohr's problem would be the following.

\begin{problem}
Let $\Omega\subset V$ be a bounded circled symmetric domain of rank $r$. Let
$f:\Omega\rightarrow\Omega$ be a holomorphic map and consider the
\emph{Schmidt decomposition} of its Taylor expansion%
\[
f(z)=\sum_{k_{1}\geq\ldots k_{r}\geq0}f_{k_{1}\cdots k_{r}}(z)
\]
(where the $f_{k_{1}\cdots k_{r}}$'s are polynomials in the irreducible
$K$-modules for the linear subgroup $K$ of $\operatorname{Aut}\Omega$).
Determine the best constant $C$ such that $z\in C\Omega$ ensures
\[
\sum_{k_{1}\geq\ldots k_{r}\geq0}\left\Vert f_{k_{1}\cdots k_{r}%
}(z)\right\Vert _{\Omega}<1
\]
for all holomorphic maps $f:\Omega\rightarrow\Omega$.
\end{problem}

{

\appendix

\section*{Appendix}%

We recall here some notations and results about complex bounded symmetric
domains and their associated Jordan triple structure (see \cite{Loos1977},
\cite{Roos2000}).

\subsection*{Bounded symmetric domains and Jordan triples}

Let $\Omega$ be an irreducible bounded circled homogeneous domain in a complex
vector space $V$. This circled realization is unique up to a linear
isomorphism. Let $K$ be the identity component of the (compact) Lie group of
(linear) automorphisms of $\Omega$ leaving $0$ fixed. Let $\omega$ be a volume
form on $V$, invariant by $K$ and by translations. Let $\mathcal{K}$ be the
Bergman kernel of $\Omega$ with respect to $\omega$, that is, the reproducing
kernel of the Hilbert space $H^{2}(\Omega,\omega)=\mathrm{Hol}(\Omega)\cap
L^{2}(\Omega,\omega)$. The Bergman metric at $z\in\Omega$ is defined by
\[
h_{z}(u,v)=\partial_{u}\overline{\partial}_{v}\log\mathcal{K}(z).
\]
The \emph{Jordan triple product} on $V$ is characterized by
\[
h_{0}(\{uvw\},t)=\partial_{u}\overline{\partial}_{v}\partial_{w}%
\overline{\partial}_{t}\log\mathcal{K}(z)\left\vert _{z=0}\right.  .
\]
The triple product $(x,y,z)\mapsto\{xyz\}$ is complex bilinear and symmetric
with respect to $(x,z)$, complex antilinear with respect to $y$. It satisfies
the\emph{\ Jordan identity}
\begin{equation}
\{xy\{uvw\}\}-\{uv\{xyw\}\}=\{\{xyu\}vw\}-\{u\{vxy\}w\}.\tag{J}%
\end{equation}
The space $V$ endowed with the triple product $\{xyz\}$ is called a
\emph{(Hermitian) Jordan triple system}. For $x,y,z\in V$, denote by $D(x,y)$
and $Q(x,z)$ the operators defined by
\begin{equation}
\{xyz\}=D(x,y)z=Q(x,z)y.\tag{A1}%
\end{equation}
The Bergman metric at $0$ is related to $D$ by
\begin{equation}
h_{0}(u,v)=\operatorname*{tr}D(u,v).\tag{A2}%
\end{equation}
A Jordan triple system is called \emph{Hermitian positive }if
$(u|v)=\operatorname*{tr}D(u,v)$ is positive definite. As the Bergman metric
of a bounded domain is always definite positive, the Jordan triple system
associated to a bounded symmetric domain is Hermitian positive.

The \emph{quadratic representation }
\[
Q:V\longrightarrow\operatorname*{End}{}_{\mathbb{R}}(V)
\]
is defined by $Q(x)y=\frac{1}{2}\{xyx\}$. The \emph{Bergman operator} $B$ is
defined by
\begin{equation}
B(x,y)=I-D(x,y)+Q(x)Q(y),\tag{A3}%
\end{equation}
where $I$ denotes the identity operator in $V$. The quadratic representation
and the Bergman operator satisfy to many identities, the most important of
which are%
\begin{align}
&  Q(Q(x)y)=Q(x)Q(y)Q(x),\tag{A4}\\
&  Q(B(x,y)z)=B(x,y)Q(z)B(y,x)\tag{A5}%
\end{align}
(see \cite{Loos1975}, \cite{Roos2000}).

The Bergman operator gets its name from the following property:
\[
h_{z}\left(  B(z,z)u,v\right)  =h_{0}(u,v)\quad(z\in\Omega;\ u,v\in V).
\]
The Bergman kernel of $\Omega$ is then given by
\[
\mathcal{K}(z)=\frac{1}{\operatorname*{vol}\Omega}\frac{1}{\det B(z,z)}.
\]

\subsection*{Spectral theory}

Let $V$ be an Hermitian positive Jordan triple system. An element $c\in V$ is
called \emph{tripotent} if $c\neq0$ and $\{ccc\}=2c$.

Two tripotents $c_{1}$ and $c_{2}$ are called \emph{orthogonal} if
$D(c_{1},c_{2})=0$. If $c_{1}$ and $c_{2}$ are orthogonal tripotents, then
$D(c_{1},c_{1)}$ and $D(c_{2},c_{2})$ commute and $c_{1}+c_{2}$ is also a tripotent.

A tripotent $c$ is called \emph{primitive} (or \emph{minimal}) if it is not
the sum of two orthogonal tripotents. A tripotent $c$ is \emph{maximal} if
there is no tripotent orthogonal to $c$.

A \emph{frame} of $V$ is a maximal sequence $(c_{1},\ldots,c_{r})$ of pairwise
orthogonal primitive tripotents. Then there exist frames for $V$. All frames
have the same number of elements, which is the \emph{rank} $r$ of $V$. The
frames of $V$ form a manifold $\mathcal{F}$, which is called the
\emph{Satake-Furstenberg boundary} of $\Omega$.

Let $V$ be a simple Hermitian positive Jordan triple system. Then any $x\in V$
can be written in a unique way
\begin{equation}
x=\lambda_{1}c_{1}+\lambda_{2}c_{2}+\cdots+\lambda_{p}c_{p}\text{,}\tag{A6}%
\end{equation}
where $\lambda_{1}>\lambda_{2}>\cdots>\lambda_{p}>0$ and $c_{1},c_{2}%
\ldots,c_{p}$ are pairwise orthogonal tripotents. The element $x$ is called
\emph{regular} iff $p=r$; then $(c_{1},c_{2},\ldots,c_{r})$ is a frame of $V$.
The decomposition (A6) is called the \emph{spectral decomposition} of $x$.

Let $\mathbf{c}=(c_{1},\ldots,c_{r})$ be a frame. For $0\leq i\leq j\leq r$,
let
\begin{equation}
V_{ij}(\mathbf{c})=\left\{  x\in V\mid D(c_{k},c_{k})x=(\delta_{i}^{k}%
+\delta_{j}^{k})x,\;1\leq k\leq r\right\}  \text{.}\tag{A7}%
\end{equation}
The decomposition
\begin{equation}
V=\bigoplus_{0\leq i\leq j\leq r}V_{ij}(\mathbf{c})\tag{A8}%
\end{equation}
is orthogonal with respect to the Hermitian product (A2) and is called the
\emph{simultaneous Peirce decomposition} with respect to the frame
$\mathbf{c}$.

Let $(p_{ij})$ be the family of orthogonal projectors onto the subspaces of
the decomposition (A8). Then, for $x=\lambda_{1}c_{1}+\lambda_{2}c_{2}%
+\cdots+\lambda_{r}c_{r}$, $\lambda_{i}\in\mathbb{R}$ ($1\leq i\leq r$) and
$\lambda_{0}=0$,%
\begin{align}
D(x,x)  & =\sum_{0\leq i\leq j\leq r}\left(  \lambda_{i}^{2}+\lambda_{j}%
^{2}\right)  p_{ij},\tag{A9}\\
Q(x)^{2}  & =\sum_{0\leq i\leq j\leq r}\lambda_{i}^{2}\lambda_{j}^{2}%
p_{ij},\tag{A10}\\
B(x,x)  & =\sum_{0\leq i\leq j\leq r}\left(  1-\lambda_{i}^{2}\right)  \left(
1-\lambda_{j}^{2}\right)  p_{ij}.\tag{A11}%
\end{align}
(See \cite[Corollary 3.15]{Loos1977}).

The map $x\mapsto\lambda_{1}$, where $x=\lambda_{1}c_{1}+\lambda_{2}%
c_{2}+\cdots+\lambda_{p}c_{p}$ is the spectral decomposition of $x$
($\lambda_{1}>\lambda_{2}>\cdots>\lambda_{p}>0$) is a norm on $V$, called the
\emph{spectral norm}. We will denote the spectral norm by $\left\Vert
x\right\Vert _{\Omega}$. It satisfies%
\begin{equation}
\left\Vert x\right\Vert _{\Omega}^{2}=\left\Vert Q(x)\right\Vert =\frac{1}%
{2}\left\Vert D(x,x)\right\Vert ,\tag{A12}%
\end{equation}
where $\left\Vert u\right\Vert $ denotes the operator norm of an $\mathbb{R}%
$-linear operator $u\in\operatorname*{End}_{\mathbb{R}}V$ with respect to the
Hermitian norm $\left\Vert x\right\Vert ^{2}=\operatorname{tr}D(x,x)$.

If $V$ has the Jordan triple structure associated to the bounded symmetric
domain $\Omega\subset V$, then $\Omega$ is the unit ball of $V$ for the
spectral norm. Conversely, if $V$ is a positive Hermitian Jordan triple, the
unit ball of the spectral norm is a bounded circled symmetric domain, whose
associated Jordan triple is $V$.

}

\end{document}